\documentclass[11pt,reqno,tbtags,draft]{amsart}
\usepackage{amssymb}
\usepackage[latin1]{inputenc}   
\usepackage{ifdraft}
\usepackage{url}
\usepackage[square,numbers]{natbib}
\usepackage{epsfig}
\usepackage{graphics,color}
\input{epsf.sty}

\numberwithin{equation}{section}

\allowdisplaybreaks

\newtheorem{theorem}{Theorem}[section]
\newtheorem{lemma}[theorem]{Lemma}
\newtheorem{proposition}[theorem]{Proposition}
\newtheorem{corollary}[theorem]{Corollary}

\newtheorem*{method}{Hamilton's method}

\theoremstyle{definition}
\newtheorem{example}[theorem]{Example}

\newtheorem{remark}[theorem]{Remark}

\newenvironment{romenumerate}[1][0pt]{
\addtolength{\leftmargini}{#1}\begin{enumerate}
 }{\end{enumerate}}

\newcounter{oldenumi}
{\setcounter{oldenumi}{\value{enumi}}
\begin{romenumerate} \setcounter{enumi}{\value{oldenumi}}}
{\end{romenumerate}}

\newcounter{thmenumerate}

\newcounter{romxenumerate}   

\newcounter{xenumerate}   

\newcommand{\refT}[1]{Theorem~\ref{#1}}
\newcommand{\refC}[1]{Corollary~\ref{#1}}
\newcommand{\refL}[1]{Lemma~\ref{#1}}
\newcommand{\refR}[1]{Remark~\ref{#1}}
\newcommand{\refS}[1]{Section~\ref{#1}}
\newcommand{\refSS}[1]{Section~\ref{#1}}

\newcommand{\refE}[1]{Example~\ref{#1}}
\newcommand{\refF}[1]{Figure~\ref{#1}}

\newcommand\marginal[1]{\ifdraft
{\marginpar[\raggedleft\tiny #1]{\raggedright\tiny #1}}
{\message{ERROR marginal requires draft option}}}

\begingroup
  \divide\count255 by 60
  \multiply\count255 by -60
  \advance\count255 by \time
  \ifnum \count255 < 10 \xdef\klockan{\the\count1.0\the\count255}
  \else\xdef\klockan{\the\count1.\the\count255}\fi
\endgroup

\newcommand\nopf{\qed}   

\newcommand{\sumim}{\sum_{i=1}^m}

\newcommand\set[1]{\ensuremath{\{#1\}}}

\newcommand\bigpar[1]{\bigl(#1\bigr)}
\newcommand\Bigpar[1]{\Bigl(#1\Bigr)}
\newcommand\biggpar[1]{\biggl(#1\biggr)}
\newcommand\lrpar[1]{\left(#1\right)}

\newcommand\xcpar[1]{\{#1\}}

\newcommand\bigabs[1]{\bigl|#1\bigr|}

\def\rompar(#1){\textup(#1\textup)}    

\newcommand\Bigparfrac[2]{\Bigpar{\frac{#1}{#2}}}

\def\xexp(#1){e^{#1}}
\newcommand\ceil[1]{\lceil#1\rceil}
\newcommand\floor[1]{\lfloor#1\rfloor}
\newcommand\fract[1]{\{#1\}}

\newcommand\mtoo{\ensuremath{{m\to\infty}}}

\newcommand\Ntoo{\ensuremath{{N\to\infty}}}

\newcommand\punkt[1]{\if.#1\else.\spacefactor1000\fi{#1}}
\newcommand\iid{i.i.d\punkt}    
\newcommand\ie{i.e\punkt}
\newcommand\eg{e.g\punkt}

\newcommand\cf{cf\punkt}

\newcommand{\aex}{a.e\punkt}

\newcommand\ii{\mathrm{i}}

\newcommand{\tend}{\longrightarrow}
\newcommand\dto{\overset{\mathrm{d}}{\tend}}
\newcommand\pto{\overset{\mathrm{p}}{\tend}}

\newcommand\bbN{\mathbb N}

\newcommand\bbQ{\mathbb Q}
\newcommand\bbZ{\mathbb Z}

\newcounter{CC}
\newcounter{cc}

\newcommand\E{\operatorname{\mathbb E{}}}
\renewcommand\P{\operatorname{\mathbb P{}}}

\newcommand\Exp{\operatorname{Exp}}
\newcommand\Po{\operatorname{Po}}

\newcommand\Be{\operatorname{Be}}

\newcommand\rise[1]{^{\overline{#1}}}
\newcommand\Modm{\operatorname{Mod}_m}

\newcommand\ga{\alpha}

\newcommand\gd{\delta}
\newcommand\gD{\Delta}

\newcommand\gG{\Gamma}

\newcommand\gl{\lambda}

\newcommand\eps{\varepsilon}

\newcommand\ett[1]{\boldsymbol1\xcpar{#1}}

\newcommand\qw{^{-1}}

\renewcommand{\=}{:=}

\newcommand\intoo{\int_0^\infty}

\newcommand\oi{[0,1]}
\newcommand\oix{[0,1)}

\newcommand\ooo{[0,\infty)}

\newcommand\setoi{\set{0,1}}

\newcommand\dd{\,\mathrm{d}}

\newcommand\pp{{\mathbf p}}
\newcommand\ud{uniformly distributed}
\newcommand\tl{\tilde\ell}
\newcommand\ttq{q}
\newcommand\tq{\tilde q}
\newcommand\hq{\hat q}
\newcommand\hS{\widehat S}
\newcommand\hl{\tl}
\newcommand\setm{\set{0,\dots,m-1}}
\newcommand\qi{^{(i)}}

\newcommand\qx[1]{^{(#1)}}
\newcommand\iij{I_j^{+}}
\newcommand\jij{I_j^{-}}

\newcommand\ppm{\ensuremath{p_1,\dots,p_m}}
\newcommand\ppom{\ensuremath{p_{(1)}\ge\dots \ge p_{(m)}}}
\newcommand\linQ{linearly independent over\/ $\bbQ$}
\newcommand\subx[1]{_{\mathrm{#1}}}
\newcommand\mux{\mu\subx}

\newcommand\Da{\ensuremath{\Delta_a}}

\newcommand\D{\ensuremath{\Delta}}

\newcommand\wi{j}
\newcommand\wj{\ensuremath{i}}
\newcommand\mins{\relax}

\newcommand\xx[1]{_{(#1)}}
\newcommand\ddxq{\dd\xq}
\newcommand\xq{\mathbf{x}}
\newcommand\bfr{\mathbf r}
\newcommand\si[1]{S_i^{#1}}
\newcommand\br{\bar r}

\newcommand\REM[1]{{\raggedright\texttt{[#1]}\par\marginal{XXX}}}

\newcommand\urladdrx[1]{{\urladdr{\def~{{\tiny$\sim$}}#1}}}

\begin{document}
\title
{The probability of the Alabama paradox}

\date{12 April, 2011; revised 14 December, 2011}

\author{Svante Janson} 
\address{Department of Mathematics, Uppsala University, 
P.O.\ Box 480,  SE-751 06, Uppsala, Sweden.}
\email{svante@math.uu.se}
\urladdrx{http://www.math.uu.se/~svante/}

\author{Svante Linusson} \thanks{Svante Linusson is a Royal Swedish Academy
  of Sciences Research Fellow supported by a grant from the Knut and Alice
  Wallenberg Foundation.} 
\address{Department of Mathematics, KTH -- Royal Institute of Technology, 
  SE-\hbox{100 44}, Stockholm, Sweden.}
\email{linusson@math.kth.se}
\urladdrx{http://www.math.kth.se/~linusson/}



\begin{abstract}
Hamilton's method 
is a natural
and common method to distribute seats proportionally between states (or
parties) in a parliament. In USA it has been  
abandoned due to some drawbacks, 
in particular the possibility of the Alabama paradox, 
but it is still in use in many other countries.
In this paper we give, under certain assumptions, a closed formula for the
asymptotic probability, as the number of seats tends to infinity, 
that the Alabama paradox occurs given the vector $p_1,\dots,p_m$
of relative sizes of the states. 

{}From the theorem we deduce a number of consequences. For example it is shown
that the  
expected number of states that will suffer from the Alabama paradox is 
asymptotically bounded above by $1/e$ and on average approximately $0.123$. 
\end{abstract}

\maketitle

\section{Introduction and main result}\label{S:intro}

Proportional representation is desired in various circumstances. One common
case is in elections in many countries, where each party is awarded a number
of seats in parliament proportional to the number of votes.
Since the number of seats has to be an integer, it becomes a mathematical
problem to choose these integers in a way that approximates exact
proportionality, and a number of different methods are in use.

In the United States, elections are done differently (with
single-member constituences), but the same mathematical problem exists
for apportionment to the House of Representatives.
By the Constitution, each state has a number of representatives
that is proportional to its population. However, the Constitution does not
specify by which method the numbers are to be determined.
(The numbers are determined by Congress every tenth year, after a new census.)
Therefore, the choice of method has been subject to much debate since 1791,
see \citet{BY} for a detailed history.
(In 1941, a specific method was chosen by law to be used not only for that
apportionment but also for all coming ones. This has eliminated the need for
debates and new choices every ten years, so there is now much less debate.
For the current method, see \cite{BY}, which also discusses why the method
is slightly biased and could be improved.)

The problem was further complicated by the fact that the Constitution does
not specify the total number of representatives. 
Thus, when discussing apportionment, the Congress discussed not only
different methods but also different sizes of the House.
During the second half of the 19th century, the favourite method
was
\emph{Hamilton's method} (also called \emph{method of largest remainder}), 
which can be described as follows.
(The method was proposed by Alexander Hamilton in 1792 for the first 
apportionment; it was then approved by Congress but vetoed by president
Washington. The method was proposed again by Samuel Vinton in 1850 when it
became law and was used, with some fiddling, for the rest of the century;
see \cite{BY} for details.)

\begin{method}
Suppose that there are $m$ states with populations $P_1,\dots,P_m$, and 
$n$ seats to distribute.
Let $P=\sumim P_i$ be the total population and let $p_i=P_i/P$ be 
the relative population of state $i$, 
\ie, its proportion of the total population. 
Calculate $\mu_i=p_i n$; this is the real number
that would give exact proportionality. First round these down and give 
$\floor{\mu_i}$ seats to state $i$. 
The sum of these numbers is almost always less
than $n$ (the exception is when all $\mu_i$ happen to be integers), and the
remaining seats are given to the states with largest remainders
$\mu_i-\floor{\mu_i}$.  In other words, $\mu_i$ is rounded up for the states
with largest remainders, and the number of states that are rounded up is
determined so that the total number of seats becomes $n$.
\end{method}

The method is simple and intuitive, and it does not bias against small or
large states.
However, in 1881 it was discovered that this method has a surprising and
unwelcome behavior when the total number of seats is changed:
It can happen that some state gets \emph{less} representatives when the
total number is increased.
More precicely, using the population figures from the 1880 census, 
a total of 299 seats would give 8 to Alabama, but a total number of 300
would give only 7 to Alabama, see \cite[p.~39]{BY} for details.
This counterintuitive behaviour got the name \emph{Alabama paradox}, and it
eventually led to the abolishment of Hamilton's method in favour of others
that do not suffer from this defect.
(The same problem was actually observed already 1871, in that case for Rhode
Island, but this went largely unnoticed \cite[p.~38]{BY}.)

\begin{remark}
The method is still used in parliamentary elections in several countries,
either to distribute seats among multi-member consituencies 
(e.g.\ Sweden), or to distribute seats among the parties  (e.g.\ Denmark).
There the number of seats is fixed in advance, so the Alabama paradox
is not an obvious problem. However, the paradox may surface and give strange
behaviour in combination with other rules, and in some election systems
where the method is used a
party that gains a vote can, in exceptional situations, actually lose a seat
in parliament, see \cite{wahlrecht}.
Germany used earlier Hamilton's method (there called Hare--Niemeyer's method)
for federal elections
but changed in 2008 for this reason (although the problem partly remains for
other reasons); the method is still used in several German
states, however.  
\end{remark}

The Alabama paradox is mathematically not strange, once it has been
noticed:
Consider three states A, B and C.
If we increase $n$ to $n+1$, the number $\mu_i=np_i$ is increased by $p_i$.
If, for example, state C is small and states A and B larger, then $\mux C$ 
increases less that $\mux A$ and $\mux B$. 
Suppose for simplicity that none of these numbers
passes an integer, so the integer parts $\floor{\mu_i}$ remains the same for
$n$ and $n+1$ for all three states; 
then their remainders $\rho_i=\mu_i-\floor{\mu_i}$ increase by
$p_i$, and it may happen that the remainders $\rho\subx A$ and $\rho\subx B$
both are 
smaller than $\rho\subx C$ when we distribute $n$ seats, but that both become
larger than $\rho\subx C$ when we increase $n$ to $n+1$. If furthermore C
had the smallest remainder that was rounded up, then the result is that 
C loses one seat while A and B gain one each. (We assume that no
other state interferes.)
A simple numerical example is given in \refF{Falabama}.

\begin{figure}
\begin{tabular}{|c|r|r|r|}
\hline
state & pop. & $\mu_i$ & seats \\
\hline
A & 53 & 5.30 & 5 \\
B & 33 & 3.30 & 3 \\
C & 14 & 1.40 & \bf 2 \\
\hline
sum & 100 & 10.00 & 10\\
\hline
\multicolumn{4}{l}{10 seats}
\end{tabular}
\qquad
\begin{tabular}{|c|r|r|r|}
\hline
state & pop. & $\mu_i$ & seats \\
\hline
A & 53 & 5.83 & \bf6 \\
B & 33 & 3.63 & \bf4 \\
C & 14 & 1.54 & 1 \\
\hline
sum & 100 & 11.00 & 11\\
\hline
\multicolumn{4}{l}{11 seats}
\end{tabular}
\caption{The Alabama paradox. Numbers in boldface are rounded up.
(Population figures may be in thousands or millions, for example.)}
\label{Falabama}
\end{figure}

So the Alabama paradox certainly may happen, and it has occurred, but how
likely is it?

Of course, the description so far is purely deterministic (except for the
necessity to draw lots sometimes when there is a tie); the Alabama
paradox either occurs or not for given parameters. Let us, however, assume 
that the population sizes are given, 
but choose a \emph{random} number $n$ of seats. 
\emph{What is the probability that the
Alabama paradox occurs if $n$ is increased to $n+1$?}

By choosing a random number $n$ we mean choosing $n$ uniformly at random
from \set{1,2,\dots,N} for some large integer $N$, and then taking the limit
(assuming that it exists) as \Ntoo.
Thus, more formally, let $s_i(n)$ be the number of seats
state $i$ receives when $n$ seats are distributed.
Increase the number of seats $n$, by one seat at a time, from 1 to $N$, and
let $\nu_i(N)$ be the number of times that state $i$ suffers from the Alabama
paradox, \ie, the number of $n< N$ such that $s_i(n+1)<s_i(n)$.
If $\nu_i(N)/N$ converges to some value $q_i$ as $N\to\infty$, we say that the
limit $q_i$ is the \emph{probability that state $i$ suffers from the Alabama
  paradox}. 
(This approach, to consider given sizes but  a random number of seats,
is used also in \cite{SJ262} where some other properties of election methods
are studied.) 

In order to calculate this limit (and show that it exists), we will make one
mathematical simplification.
Recall that a set \set{x_1,\dots,x_k} of real numbers is 
\emph{\linQ} if there is no relation
$a_1x_1+\dots+a_kx_k=0$ with all coefficients $a_i$ rational and not all $0$.
(Equivalently, there is no such relation with integer coefficients $a_i$,
not all 0.)
We will assume that the relative population sizes are \linQ.

\begin{remark}\label{Rind}
Mathematically, this assumption is reasonable, since if we choose $\ppm$ at
random (uniformly given that their sum is 1), they will almost surely be \linQ.
However, for the practical problem of apportionment,
the assumption is clearly unreasonable since the populations $P_i$ are
integers and the $p_i$ thus rational numbers. Nevertheless, 
the formula below is a good approximation 
if the numbers $p_i$ have large denominators and there are no relations
\begin{equation}\label{rind}
 a_1p_1+\dots+a_mp_m=0 
\end{equation}
with small integers $a_1,\dots,a_m$.
More precisely, it will be shown in \refSS{SSdep} that for any
$\eps>0$, there is an $A=A(m,\eps)$ such that the value $\ttq_i$ in 
\eqref{ta} or \eqref{tab} differs
from the exact probability   by less than $\eps$ 
for every distribution $(p_i)_1^m$ for which there is no such relation
\eqref{rind}
with integers $a_i$ and $\sum_i|a_i|\le A$; we omit the details.
\end{remark}

We leave it as an open problem to extend the result below and find exact
formulas for all \ppm, and in particular for rational \ppm.
(If $\ppm$ are rational, then the sequence $s_i(n+1)-s_i(n)$ will be
periodic, so the limit $q_i$ certainly exists; the existence in general is
shown in \refSS{SSdep}.)
Note that some modifications are required for rational \ppm.
For example with three states and $p_1=p_2=2/5$, $p_3=1/5$, it is easily
seen that the Alabama paradox never occurs, so the probability is 0 for all
three states. See also Proposition \ref{P:highE}, were it is shown that the
expected number of states  
suffering from the paradox could be arbitrarily close to 1.

We use the standard notations $x_+\=\max\set{x,0}$ and
$x_-\=(-x)_+=-\min\set{x,0}$, noting $x=x_+-x_-$ and $|x|=x_++x_-$. 
Let $e_k(x_1,\dots,x_n)$ denote the \emph{elementary symmetric polynomial}
of degree $k$ in $n$ variables, i.e.\ 
$e_k(x_1,\dots,x_n)\=\sum_{1\le i_1<\dots<i_k\le n} \prod_{j=1}^k
x_{i_j}$. With $\Be(p)$ we mean the Bernoulli distribution;
thus $I\sim\Be(p)$ if $\P(I=0)=1-p$ and $\P(I=1)=p$.
We let $\ppom$ be the population vector $\ppm$
rearranged in increasing order, and let  $q\xx1,\dots, q\xx m$ 
be the corresponding probabilities of the Alabama paradox, which by
\refC{Cmin} below is the vector $q_1,\dots,q_m$ rearranged in increasing order.
For clarity, we will use the notation $p\xx i$ and $q\xx i$
whenever we consider the
states in increasing order, and $p_i,\,q_i$ only when the order is irrelevant.

\begin{theorem}\label{TA}
Suppose that $m$ states have relative sizes $\ppm$, with $\sumim p_i=1$,
and assume that $\ppm$ are \linQ.
Then the probability $q_i$
that state $i$ suffers from the Alabama paradox when we
increase the total number of seats by one equals
\begin{equation}\label{ta}
\ttq_i\=
\frac1m \E\bigpar{\si--\si+-1}_+,  
\end{equation}
where
$\si+=\sum_{j:p_j<p_i} I_j\qi$ and 
$\si-=\sum_{j:p_j>p_i} I_j\qi$ with $I_j\qi\sim\Be(|p_i-p_j|)$
and $I_1\qi,\dots,I_m\qi$ independent.
More explicitly, if the states are ordered with 
$\ppom$, this can be written
\begin{equation}\label{tab}
\ttq_{(i)}=
\frac{1}{m}\sum_{s=0}^{m-i}\sum_{k=2}^{i-1}(-1)^{s+k}
\binom{s+k-2}{s}
e_{k}(\br_{1}^{(i)},\dots,\br_{i-1}^{(i)})
e_s(\br_{i+1}^{(i)},\dots,\br_{m}^{(i)}),
\end{equation}
where $\br_j^{(i)}:=|p_{(i)}-p_{(j)}|$.
\end{theorem}

In other words, each $I_j\qi\in\setoi$ with
$\P(I_j\qi=1)=|p_i-p_j|$. 

\begin{remark}
If we do not order the states, then equation \eqref{tab} can equivalently be written as
\begin{equation}\label{tabb}
\ttq_i=
\frac{1}{m}\sum_{s=0}^{m-3}\sum_{k=2}^{m-s-1}(-1)^{s+k}
\binom{s+k-2}{s}e_s\bigpar{r_{1+}^{(i)},\dots,r_{m+}^{(i)}}
e_{k}\bigpar{r_{1-}^{(i)},\dots,r_{m-}^{(i)}},
\end{equation}
where $r_{j\pm}^{(i)}:=(p_i-p_j)_\pm$.
Since \eqref{tabb} is symmetric under permutations of $\ppm$, we may  
assume $\ppom$. 
In this case, $r\qi_{j+}=0$ for $j\le i$ and $r\qi_{j-}=0$ for $j\ge i$, and
it is easily seen that the sums in \eqref{tabb} and \eqref{tab} are equal.
\end{remark}

The proof of the theorem is given in \refS{Spf}. 
We give first several consequences of the main theorem in 
Sections \ref{Scor}--\ref{Srandom};
the proofs of these results are given in Sections \ref{SpfC}--\ref{Spfrandom}. 

For simplicity, we have here considered one state at a time.
It may happen that the Alabama paradox occurs for two (or more) states at
the same time, although this is less likely, see \refSS{SSseveral}.

\section{Further results}\label{Scor}

In the case of three states, \refT{TA} yields the following simple formula.

\begin{corollary}\label{C3}
Suppose that there are three states with relative sizes 
$p_{(1)}\ge p_{(2)}\ge p_{(3)}$ with
$p_{(1)}+p_{(2)}+p_{(3)}=1$,  
and assume that $p_{(1)},p_{(2)},p_{(3)}$ are \linQ.
Then only the smallest state can suffer from the Alabama paradox, and the
probability of this is
$\frac13(p_{(1)}-p_{(3)})(p_{(2)}-p_{(3)})$.
\end{corollary}
The supremum of this probability over all distributions
$(p_{(1)},p_{(2)},p_{(3)})$ 
is $1/12$, which is never attained but is approached in the extreme case
when $p_{(3)}$ is very small and $p_{(1)}$ and $p_{(2)}$ both are close to $1/2$.

In general, the Alabama paradox can
affect any state except the two largest, but it is much more likely to
affect small states. 

\begin{remark}
Note however that Hamilton's method is unbiased. On the
average, state $i$ increases its number of seats by $p_i$ each time $n$ is
increased, so if it sometimes suffers from the Alabama paradox and its
representation decreases with frequency $q_i$, 
this must be compensated by a frequency $p_i+q_i$ of the times when the
number of seats increases.   
\end{remark}

\begin{corollary}\label{Cmin}
  In addition to the assumptions of \refT{TA}, assume that $\ppom$.
Then $\ttq_{(m)}\ge \ttq_{(m-1)}\ge\dots\ge \ttq_{(3)}\ge
 \ttq_{(2)}=\ttq_{(1)}=0$. 
Moreover, the largest 
probability is
\begin{equation}\label{cmin1}
  \ttq_{(m)}=\frac1m\prod_{j=1}^{m-1}\bigpar{1-(p_{(j)}-p_{(m)})}-p_{(m)}.
\end{equation}
We have, for any $i$, the inequalities
\begin{equation}\label{cmin2}
  \frac1m\Bigpar{e^{-1}-mp_{(i)}-\frac12\sum_j p_j^2}
\le \ttq_{(i)}
<
\frac1m\Bigpar{1-\frac{1}{m-1}}^{m-1}<\frac1m e^{-1}.
\end{equation}
\end{corollary}

If $p_{(m)}\to0$ and all other $p_{(j)}\to1/(m-1)$, then \eqref{cmin1} shows that 
$\ttq_{(m)}\to\frac1m\bigpar{1-\frac{1}{m-1}}^{m-1}$, so this is, for a given
$m$,  the least upper bound on the probability of the Alabama paradox
for a specific state in the linearly independent case
(but only in that case, see \refSS{SSdep}), which 
for large $m$ approaches $1/me$.

\begin{corollary}\label{Ce}
  Under the assumptions of \refT{TA}, the expected number of states
  suffering from the Alabama paradox each time the number of seats is
  increased is less than $1/e$.
This bound is approached 
 if we let $m\to\infty$ and suppose that $m-o(m)$ of the states are
very small, with $p_i=o(1/m)$, and that the remaining states are medium-size
with $p_i=o(1)$.
\end{corollary}
In this extremal case, the paradox is thus very common. It can be even more
common in the rational case, 
see Proposition \ref{P:highE}.
See also \refE{Efluid}, where we show that the probability of at least
one state suffering the Alabama paradox in this case converges to
$1-2/e\approx0.264$. 
We conjecture that this is the upper bound of the probability that at least
one state suffers the paradox (under the assumption of \refT{TA}); 
note that \refC{Ce} shows that the
probability is always less than $e\qw\approx0.368$.

The exact formula in \eqref{tab} is unwieldy when $m$ is large; it may then be
attractive to use 
\eqref{ta} with
a Poisson approximation of $S^\pm_i$:

\begin{corollary}
  \label{CP}
Under the assumptions of \refT{TA}, let (for a given $i$)
$\gl^+\=\sum_{j:p_j<p_i}(p_i-p_j)$ and $\gl^-\=\sum_{j:p_j>p_i}(p_j-p_i)$,
Let further $\hS^+\sim\Po(\gl^+)$ and $\hS^-\sim\Po(\gl^-)$ be
independent Poisson random variables, and define
\begin{align}\label{cp}
\hq_i&:=
\frac1m \E\bigpar{\hS^--\hS^+-1}_+
\\&\phantom:
=\frac1m\sum_{j\ge k+2} (j-k-1)\frac{(\gl^-)^j(\gl^+)^k}{j!\,k!}
e^{-\gl^--\gl^+}
\label{cp+}
\\&\phantom:
=\frac1m\sum_{j\le k} (k+1-j)\frac{(\gl^-)^j(\gl^+)^k}{j!\,k!}
e^{-\gl^--\gl^+}-p_i.
\label{cp-}
\end{align}
Then
\begin{equation}\label{cp2}
  |q_i-\hq_i|\le\frac1m\sum_{j=1}^m(p_j-p_i)^2.
\end{equation}
\end{corollary}

\section{Average probability of the Alabama paradox}\label{Srandom}

We have so far considered the probability of the Alabama paradox for given
relative population sizes $p_1,\dots,p_m$. Let us now instead fix $m\ge3$ and
consider the average probability over all population distributions. In other
words, in this section we
let $(p_1,\dots,p_m)$ be random and uniformly distributed 
over the simplex 
$\gD^m\=\set{(\ppm)\in\oi^m:\sum_ip_i=1}$, and take the expectation.
Note that then $\ppm$ are \linQ{} a.s., so we may in the sequel
assume that \refT{TA} and its corollaries apply.

As above,  $\ppom$ denotes the population vector $\ppm$
rearranged in increasing order;
note that $(p\xx1,\dots,p\xx m)$ is uniformly distributed over the subset
$\ppom$ of the simplex $\gD^m$.
Similarly,  $q\xx1\le\dots\le q\xx m$
are the corresponding probabilities of the Alabama paradox, which by
\refC{Cmin} are  $q_1,\dots,q_m$ rearranged in increasing order.
In particular, 
$q\xx m$ is the probability of the Alabama paradox for the smallest state.
Note that $q\xx{1}=q\xx 2=0$ but a.s.\ $q\xx k>0$ for $k\ge 3$.
Since $q_i$ and $q\xx i$ are functions of $\ppm$, they too now 
are random variables, and we may ask for their expectations, or other
properties of their distributions.
We use the notations $\dto$ and $\pto$ for convergence in distribution and
probability, respectively, always as \mtoo.

In the case $m=3$, the average of the formula in \refC{C3} over all
$(p\xx1,p\xx2,p\xx3)$ is easily found by integration, for 
example by the substitution $p\xx1=1-p\xx3-p\xx2$ and integrating over
$(p\xx2,p\xx3)$ 
with the conditions $0<p\xx3<\frac13$ and $p\xx3<p\xx2<(1-p\xx3)/2$; a
calculation 
yields the probability $\E q\xx 3$ for the Alabama 
paradox for three states of random sizes as $1/36$.

We extend this to larger $m$. First we consider only the smallest state,
which is most likely to experience the paradox. Recall the notation 
$m\rise k:=m(m+1)\dots (m+k-1)$ 
for the rising factorial.

\begin{theorem} \label{T:expected}
With uniformly random relative population sizes, 
the expected probability 
$\E q\xx m$ that the smallest state among $m$ states will suffer from the
Alabama paradox is 
\begin{equation}
  \label{texp}
\E q\xx{m} =
\frac1m\sum_{k=0}^{m-1} (-1)^{k}\frac{\binom{m-1}{k}}{m\rise k} 
-\frac{1}{m^2}
=
\frac1m\sum_{k=2}^{m-1} (-1)^{k}\frac{\binom{m-1}{k}}{m\rise k} 
\end{equation}
Hence, as \mtoo, 
\begin{equation}\label{texpoo}
\E q\xx{m} 
= \frac{e^{-1}}{m}-\frac{1}{m^2}+O\left(\frac{1}{m^3}\right).  
\end{equation}
Furthermore, 
$mq\xx m\pto e\qw$
as \mtoo.
\end{theorem}

Hence, for large $m$, the expected value of the largest probability $q\xx m$ 
is asymptotically equal to the upper bound $e\qw/m$ given in \refC{Cmin}, 
and, furthermore,  $q\xx m$ is close to $e\qw/m$ for most $\ppm$.

\begin{remark}
  The sum in \eqref{texp} is a hypergeometric sum and the result can be
  written
\begin{equation}
\E q\xx{m} =
\frac1m{}_1F_1(1-m;m;1)
-\frac{1}{m^2}
\end{equation}
with a confluent hypergeometric function ${}_1F_1$ (in this case a polynomial).
\end{remark}

For small $m$ we have the following table.
\smallskip
\begin{center}
\begin{tabular}{c|c|c|c|c|c|c|c|c|c|c}
$m$& 2&  3 & 4 & 5 & 6 & 7 & 8 & 9 \\
\hline\hbox{\vrule width0pt height 12pt}
$\E q\xx m$
&$0$&$\frac1{36}$&${\frac {17}{480}}$&${\frac {61}{1680}}$&
${\frac {907}{25920}}$&${\frac {153709}{4656960}}$&${\frac{855383}{27675648}}$
&${\frac {134964353}{4670265600}}$
\end{tabular}  
\end{center}
\medskip

For a fixed (or random) state, we obtain a more complicated formula.

\begin{theorem}
  \label{TE}
With uniformly random relative population sizes
the average probability $\E q_m$  that a given state among $m$ states
will suffer from the Alabama paradox is
\begin{multline*}
\E q_m=
\frac{1}{m}\sum_{s=0}^{m-3}\sum_{k=2}^{m-s-1}
\sum_{i=0}^{s}  
\sum_{j=0}^{s-i}
(-1)^{k+i+j}\binom{s+k-2}{s} 
\\ \times
\frac{(m-1)!^2}{k!\,i!\,(s-i-j)!\,
(m-1-k-s)!\,(m-1+k+s)!} 
{(i+k+1)^{-j-1}}.
\end{multline*}
\end{theorem}

For small $m$ we have the following table.
\smallskip
\begin{center}
\begin{tabular}{c|c|c|c|c|c|c|c|c|c|c}
$m$& 2&  3 & 4 & 5 & 6 & 7 & 8 \\
\hline\hbox{\vrule width0pt height 12pt}
$\E q_m$
&$0$
&${\frac {1}{108}}$&${\frac {17}{1440}}$&${\frac {523}{43200}}$&${\frac {
2287039}{195955200}}$&${\frac {100704757}{9144576000}}$&${\frac {
404675341849}{39230231040000}}$
\end{tabular}  
\end{center}
\medskip

The average probability $\E q_m$ in \refT{TE} is, of course, always smaller
than the average of the largest probability $\E q\xx m$ in \refT{T:expected}.
It is somewhat surprising that it is not much smaller, the ratio is close to
$1/3$, as is shown for some small $m$ by the following table (with rounded values
 computed by Maple):

\smallskip
\begin{center}
\begin{tabular}{c|c|c|c|c|c|c|c|c|c|c}
$m$&   3 & 10 & 20 & 30 & 50 &100 \\
\hline
$\E q_m/\E q\xx m$
&$0.33333$&$ 0.33392$&$0.33441$&$ 0.33457$&$0.33474 $&$0.33487$
\end{tabular}  
\end{center}

Indeed, this persists for large $m$, and $\E q_m$ is really of order $1/m$,
since \refT{TE} implies the following asymptotic formula by dominated
convergence: 

\begin{corollary}
  \label{CE}
With uniformly random relative population sizes, 
the expected number of occurrences of the Alabama paradox among all $m$
states is $m\E q_m$ which 
as \mtoo{} has the limit 
\begin{equation*}
 m\E q_m\to 
b\=\sum_{s=0}^{m-3}\sum_{k=2}^{m-s-1}
\sum_{i=0}^{s}  
\sum_{j=0}^{s-i}
(-1)^{k+i+j}\binom{s+k-2}{s} 
\frac{(i+k+1)^{-j-1}}{k!\,i!\,(s-i-j)!} .
\end{equation*}
\end{corollary}

Thus, $\E q_m\sim b/m$. 
We do not know any better closed form for $b$, but numerically we obtain
from Maple 
$b\approx 0.12324$ 
and thus, using \refT{T:expected},
 $\E q_m/\E q\xx m\to be\approx 0.33501$, 
in accordance with the table
above. (If Maple is correct, the limit is not exactly $1/3$, but quite
close.) 

The formula for $b$ in \refC{CE} as an alternating quadruple sum is not very
illuminating, and it is not even easy to see that $b>0$ from it, but that at
least follows from the alternative representation in the next theorem, which
adds more information on the asymptotic relation between size and
probability for the Alabama paradox for random populations.

\begin{theorem}
  \label{Too}
Define, for any $\gl^-,\gl^+\ge0$,
with $\hS^\pm\sim\Po(\gl^\pm)$ independent,
similarly to \eqref{cp}--\eqref{cp-},
\begin{align}\label{xcp}
\Phi(\gl^-,\gl^+)&:=
 \E\bigpar{\hS^--\hS^+-1}_+
\\&\phantom:
=\sum_{j\ge k+2} (j-k-1)\frac{(\gl^-)^j(\gl^+)^k}{j!\,k!}
e^{-\gl^--\gl^+}
\label{xcp+}
\\&\phantom:
=\sum_{j\le k} (k+1-j)\frac{(\gl^-)^j(\gl^+)^k}{j!\,k!}
e^{-\gl^--\gl^+}+\gl^--\gl^+-1,
\label{xcp-}
\end{align}
and $\Psi(x)\=\Phi\bigpar{e^{-x},e^{-x}-1+x}$.
Then, as \mtoo,
with random population sizes,
\begin{equation}\label{too}
  (mp_m,mq_m)\dto\bigpar{T,\Psi(T)},
\end{equation}
where $T\sim\Exp(1)$.
In particular,
\begin{equation}\label{tooe}
  \begin{split}
  m\E q_m&\to b=\E \Psi(T)
\\&
\quad=
\sum_{j\ge k+2} (j-k-1)\intoo\frac{e^{-jx}(e^{-x}-1+x)^k}{j!\,k!}
e^{-2e^{-x}-2x+1}\dd x.	
  \end{split}
\end{equation}
\end{theorem}

By \refT{Too} and its proof, 
we can further say that, 
assuming random, uniformly distributed  populations, 
a state with $p_i=x/m$
has probability $q_i\approx \Psi(x)/m$ of suffering from the Alabama
paradox.
Note that the extreme case in
\refT{T:expected} can be seen (informally) as the limiting case $x=0$,
since we have $mp\xx m\pto0$ and $mq\xx m\to e^{-1}$, 
and indeed $\Psi(0)=\Phi(1,0)=e^{-1}$, e.g.\ by \eqref{xcp-}.

We plot the function $\Psi$ in \refF{FPsi}.

\epsfysize=8cm
\begin{figure}
 \begin{center} 
 \epsfbox{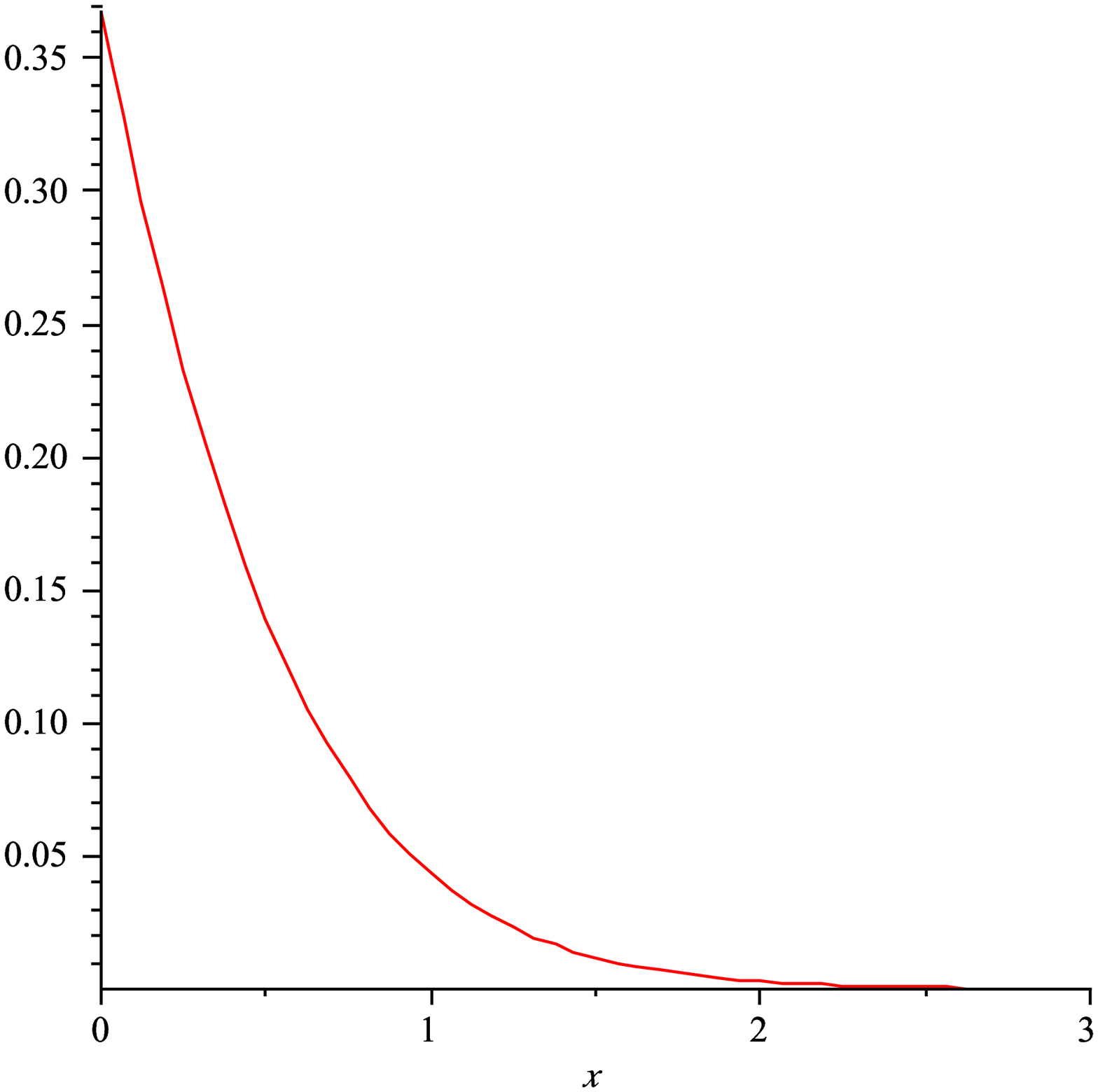}
\end{center}
\caption{The function $\Psi$ in \refT{Too}.}\label{FPsi}
\end{figure}

\section{Proof of \refT{TA}}\label{Spf}
We analyse the process of successive distributions of seats as follows.

Think of the different states as $m$ runners on a circular track, with state
$i$ running at constant speed $p_i$ (laps/time unit).
At time $n$, state $i$ has run a distance $np_i$, and thus
$\floor{np_i}$ full laps, so the number of
seats that it gets is the number of completed laps, plus an additional
seat for each of the states that have come furthest on the next lap; the
number of 
these additional seats is chosen such that the total number of seats is $n$.

We reformulate this by moving the finishing line; we mark its position by a
flag and count laps as runners pass the flag. 
We place the flag by the runner that got the last additional seat, \ie{} the
state with the smallest remainder that is rounded up. 
Then the number of seats a state gets equals its number of
laps, for every state.
(The assumption that $\ppm$ are \linQ{} implies that ties cannot occur, so
we do not have to worry at all about ties in this proof.) 

Let us increase the total number of seats from $n$ to $n+1$ in two steps.
We first increase time from $n$ to $n+1$ continuously, letting the runners
run, but at the same time we also move the flag, by letting it be carried by
a runner, so that the total number of laps stays at $n$.
This means that if the runner carrying the flag overtakes another, slower,
runner, then the flag is passed to the slower runner and both runners keep
the same numbers of laps. On the other hand, if the runner carrying the flag
is overtaken by a faster runner, then the flag is passed to the faster
runner, who gets one lap more, while the former flag-holder loses one lap.
(Other overtakings do not affect the flag, nor the number of laps for
anyone.)

Finally, at time $n+1$ we increase the total number of seats by one; this
means that the runner carrying the flag throws it to the next runner behind
him (her), who gains another lap.

We count positions, at any given time, relative to the flag
and say that
position 0 is the runner carrying the flag,
positions $-1,-2,\dots$ are the runners behind the flag-carrying runner,
and positions $1,2,3,\dots$ the runners in front of him/her.
(Since the track is circular, position $k$ and position $k-m$ are the same,
but that does not matter as long as we take a little care.)
It is easy to see that when one runner overtakes another, their positions
(which differ by 1) are exchanged, while all other positions remain the
same; this hold also if one of them carries the flag.

Consider a specific runner, say runner $i$.
The position of $i$  increases by 1 each time $i$ overtakes
someone else, and it decreases by 1 each time $i$ is overtaken. 
Furthermore, it increases by 1 at the final step when the flag is thrown.
Thus, if $S^+$ is the number of runners overtaken by $i$, and $S^-$ the number
of runners overtaking $i$, during the interval $[n,n+1]$, then the position
is increased by $S^+-S^-+1$.
Since the number of laps is changed only when the position changes between 0
and $-1$, we see that:
\begin{romenumerate}
\item[(+)] 
State $i$ gains a seat if $S^+-S^-+1>0$ and runner $i$ has at time $n$
  one of the positions $-1, -2, \dots,-(S^+-S^-+1)$.
\item[($-$)] 
State $i$ loses a seat if $S^+-S^-+1<0$ and runner $i$ has at time $n$
  one of the positions $0, 1, \dots,|S^+-S^-+1|-1$.
\end{romenumerate}
Case $(-)$ is thus when the Alabama paradox occurs for state $i$.
Let $L$
be the position of  runner $i$  relative to the flag at time $n$,
normalized to have $L\in\set{0,\dots,m-1}$.
Then the Alabama paradox occurs if and only if
\begin{equation}\label{ala}
  S^+-S^-+1<-L.
\end{equation}

Let the indicator $\iij$ be $1$ if $i$ overtakes $j$ during $[n,n+1]$, and 0
otherwise; similarly, let $\jij$ be 1 if $i$ is overtaken by $j$ and 0
otherwise.
Then $S^+=\sum_j \iij$ and $S^-=\sum_j \jij$.
(Note that no runner can overtake another more than once during $[n,n+1]$.)
We let $\fract x\= x-\floor x\in[0,1)$ denote the fractional part of a real
number $x$. 
(We will also use $\{\}$ to denote sets;  the meaning should be clear from
the context.) 
Then
\begin{align}
\iij=1 \iff p_i&>p_j \text{ and } 0<\fract{np_j-np_i} < p_i-p_j,
\label{i+}
\\
\jij=1 \iff p_i&<p_j \text{ and } 1-(p_j-p_i)<\fract{np_j-np_i} < 1.
\label{i-}
\end{align}
We calculate the probability of \eqref{ala} by finding the asymptotic joint
distribution of $L$ and
the fractional parts $\fract{np_j-np_i}$, $j\neq i$, 
where again we choose $n$ uniformly at
random with $1\le n\le N$, and then let \Ntoo.
By the formulas above, this gives the  asymptotic joint
distribution of $S^+$, $S^-$ and $L$, and thus
the (asymptotic) probability of \eqref{ala}.

We say that an infinite sequence $(v_n)_{n\ge1}\in\oix^{m-1}\times\setm$ is
\emph{uniformly distributed} if the empirical distributions
$N\qw\sum_{n=1}^N\gd_{v_n}$
converge to the uniform distribution as \Ntoo, where $\gd_{v_n}$ denotes the Dirac measure. 
This means that if $A\subseteq\oix^{m-1}$ with $\gl(\partial A)=0$
and $k\in\setm$, then
$\#\set{n\le N:v_n\in A\times\set{k}}/N\to \gl(A)/m$. 
(Here $\gl$ is the usual Lebesgue measure.)
This is a simple extension of the standard notion of uniform distribution
for a sequence in $\oix^{m-1}$.
We claim the following. (For notational convenience, we state the case
$i=1$ only.)
\begin{lemma}\label{Lu}
  Suppose that $p_1,\dots,p_m$ are \linQ, and let 
$L_n\in\set{0,\dots,m-1}$ 
be the position of  runner $1$ relative to the flag at time $n$.
Then the sequence of
vectors $v_n=(\fract{n(p_2-p_1)},\dots,\fract{n(p_{m}-p_1)},L_n)$, $n\ge1$,
is uniformly distributed on $[0,1)^{m-1}\times\set{0,\dots,m-1}$.
\end{lemma}

We postpone the proof and first complete the proof of \eqref{ta}.

By \eqref{i+} and \eqref{i-}, for each $j$, at most one of $\iij$ and $\jij$
is non-zero, depending on whether $p_j<p_i$ or $p_j>p_i$.
We simplify the notation by letting $I_j=\iij+\jij$; thus 
$S^+=\sum_{j:p_j<p_i} I_j$ and 
$S^-=\sum_{j:p_j>p_i} I_j$.

For a given $N$, these are random variables, and we have, letting $\P_N$
denote the probability when $n$ is uniformly chosen with $n\le N$,
\begin{equation}\label{pala}
\P_N(S^+-S^-+1<-L)=
\sum_{\ell=0}^{m-1}\P_N(L=\ell\text{ and }S^--S^+-1>\ell).
\end{equation}
As \Ntoo, \refL{Lu} and \eqref{i+}--\eqref{i-} show that the distribution
of $L=L_n$ converges to the uniform distribution on $\setm$ and the distribution
of $I_j$ converges to $\Be(|p_i-p_j|)$ for all $j\neq i$;
moreover, these are asymptotically independent.
Hence, \eqref{pala} yields 
\begin{equation*}
  \begin{split}
\P_N(S^+-S^-+1<-L)
&\to
\sum_{\ell=0}^{m-1}\frac1m\P(\si--\si+-1>\ell)
\\&
=\frac1m \E\bigpar{\si--\si+-1}_+,	
  \end{split}
\end{equation*}
where 
$\si+=\sum_{j:p_j<p_i} I_j\qi$ and 
$\si-=\sum_{j:p_j>p_i} I_j\qi$ with $I_j\qi\sim\Be(|p_i-p_j|)$ independent.
This is the result stated in \eqref{ta}.

We proceed to show \refL{Lu}. 
First recall a  well-known result by Weyl.
(The standard proof is by showing that the Fourier transform (characteristic
function) 
$\frac1N\sum_{n=1}^N \exp\bigpar{2\pi\ii\sum_{j=1}^k n_j\fract{ny_j}}\to0$,
as \Ntoo, for any fixed integers $n_1,\dots,n_k$, not all 0,
see for example \cite[Exercises 3.4.2--3]{Grafakos}.)

\begin{lemma}[Weyl]\label{LW}
Suppose that $y_1,\dots,y_k$ and\/ $1$ are \linQ. 
Then the sequence of vectors 
$(\fract{ny_1},  \dots,\fract{ny_k})\in[0,1)^k$ 
is uniformly distributed in $[0,1)^k$.
\nopf
\end{lemma}

We will need the following extension.
Let $\Modm(x)=m\fract{x/m}$; this is the remainder when $x$ is divided by
$m$. Thus, if $r$ is an integer, then $\Modm(r)$ is the unique integer in
$\setm$ such that $\Modm(r)\equiv r \pmod m$.

\begin{lemma}\label{Lux}
Suppose that $y_1,\dots,y_k$ and\/ $1$ are \linQ. 
Let $\ell_n=\Modm\bigpar{ \bigpar{\sum_{j=1}^k\floor{ny_j}}-n}\in\set{0,\dots,m-1}$. 
Then the sequence of vectors 
$(\fract{ny_1},  \dots,\fract{ny_k}, \ell_n)\in[0,1)^k\times\set{0,\dots,m-1}$ 
is uniformly distributed in $[0,1)^k\times\set{0,\dots,m-1}$.
\end{lemma}

\begin{proof}
%
Let $z_j=y_j/m$ and $w_n=(w_{n1},\dots,w_{nk})$ with $w_{nj}=\fract{nz_j}$,
$j=1,\dots,k$. Then $z_1,\dots,z_k$ and 1 are \linQ, and thus 
\refL{LW} (Weyl's theorem) shows that the sequence $(w_n)_{n\ge1}$ is uniformly
distributed in $\oix^k$. 
Further,
\begin{equation*}
 ny_j-mw_{nj}=mnz_j-m\fract{nz_j}=m\floor{nz_j}\equiv0\pmod m. 
\end{equation*}
Hence,
\begin{align}
  \fract{ny_j}=\fract{mw_{nj}}
\qquad\text{and}\qquad
\floor{ny_j}\equiv\floor{mw_{nj}}\pmod m.
\end{align}
Thus, $\Modm(\floor{ny_j})=\floor{mw_{nj}}$.

Let $\hl_{nj}=\Modm(\floor{ny_j})=\floor{mw_{nj}}$.
If a sequence $(u_n)$ is uniformly distributed in $\oix$, then $(mu_n)$ is
uniformly distributed in $[0,m)$ and the vectors 
$(\fract{mu_n},\floor{mu_n})$
are \ud{} in $\oix\times\setm$. Using this argument in each coordinate,
the fact that $(w_n)$ is \ud{} in $\oix^k$ implies that the sequence
of vectors 
$(\fract{ny_j},\hl_{nj})_{j=1}^k=(\fract{mw_{nj}},\floor{mw_{nj}})_{j=1}^k$
is \ud{} in $\oix^k\times\setm^k$.

Let $\tl_n=\Modm\bigpar{\sum_{j=1}^k\floor{ny_j}}$.
Then $\tl_n=\Modm\bigpar{\sum_{j=1}^k\hl_{nj}}$, and it follows that the
sequence of vectors 
$(\fract{ny_1},  \dots,\fract{ny_k}, \tl_n)$ 
is uniformly distributed in $[0,1)^k\times\set{0,\dots,m-1}$.

This is almost what we claim. To complete the proof, we consider a
subsequence of the form $n=m\nu+n_0$, $\nu\ge1$. 
Weyl's theorem holds for such
subsequences too, as a consequence of the version in \refL{LW} since
$\fract{(m\nu+n_0)p_j}=\fract{\nu(mp_j)+n_0p_j}$ where $mp_1,\dots,mp_k$ and 1
are \linQ, and the constant shift by $n_0p_j$ does not affect the uniform
distribution. Consequently, the argument above shows that  
$(\fract{ny_1},  \dots,\fract{ny_k}, \tl_n)$ is \ud{} for each such
subsequence. But along the subsequence, $\ell_n=\Modm(\tl_n-n_0)$, so 
$(\fract{ny_1},  \dots,\fract{ny_k}, \ell_n)$ is  \ud{} for each such
subsequence, and thus for the entire sequence.
\end{proof}

\begin{proof}[Proof of \refL{Lu}]
Suppose that runner 1 carries the flag at time $n$, \ie, $L_n=0$. 
Then state $1$ gets an
additional seat, \ie, its number of seats is rounded up to $\ceil{np_1}$,
and state $j$ gets
\begin{equation*}
  \ceil{np_j-\fract{np_1}}
=
  \ceil{np_j-{np_1}}+\floor{np_1}
=  \floor{np_j-{np_1}}+\ceil{np_1}
\end{equation*}
seats.
Since the total number of seats is $n$, we have, still in the case $L_n=0$,
\begin{equation*}
  n=\sum_{j=2}^m \floor{np_j-{np_1}}+m\ceil{np_1}
\equiv\sum_{j=2}^m \floor{np_j-{np_1}}\pmod m.
\end{equation*}
In general, there are $L_n$ additional states whose numbers of
seats are rounded up (or $-L_n$ fewer, if $L_n<0$); thus 
\begin{equation*}
  n
\equiv L_n+\sum_{j=2}^m \floor{np_j-np_1}\pmod m
\end{equation*}
and
\begin{equation}\label{l}
 L_n
\equiv n-\sum_{j=2}^m \floor{n(p_j-p_1)}\pmod m.
\end{equation}

Let $y_j=p_j-p_1$, $j=2,\dots,m$. Since \ppm{} are \linQ, 
it is easily seen that $y_2,\dots,y_m$ and $\sum_1^m p_j=1$ also are \linQ.
(This can be seen as a change of basis, using a non-singular integer matrix,
in a vector space of dimension $m$ over $\bbQ$.)
Thus \refL{Lu} follows from \refL{Lux} 
(with $k=m-1$, after renumbering $y_2,\dots,y_m$),
since $L_n\equiv -\ell_n\pmod m$ by \eqref{l}.
\end{proof}

This completes the proof of \eqref{ta}.
We proceed to show that this can be evaluated as \eqref{tab}. 

We may assume that the states are ordered by size,
$p_1\ge\dots \ge p_m$; thus 
$p_i=p\xx i$. 
Consider the $\wj$:th largest state.  
Let $X_-\subseteq [\wj-1]$ and $X_+\subseteq \{\wj+1,\dots,m\}$, 
where as usual $[n]:=\{1,\dots,n\}$. 
We think of
$X_-$ and $X_+$ as the indices $\wi$ for which $I_\wi^{(\wj)}=1$. Let
$P_=(X_-,X_+)$ be the probability that $\set{\wi:I_\wi^{(\wj)}=1}= X_-\cup
X_+$.  Then clearly, 
for simplicity writing $\br_j\=\br_j\qi=|p_i-p_j|$,
\[P_=(X_-,X_+)=\prod_{\wi\in X_-\cup X_+} \br_\wi \prod_{\wi\in [m]\backslash (X_-\cup X_+\cup \{\wj\})} (1-\br_\wi).
\]

The formula \eqref{ta}
for the probability that the $\wj$:th largest state suffers
from the Alabama paradox is thus 
\begin{multline*}
m\ttq\xx \wj=
\\\sum_{X_-\subseteq [\wj-1],X_+\subseteq \{\wj+1,\dots,m\}} (|X_-|-|X_+|-1)\;\cdot 
\prod_{\wi\in X_-\cup X_+} \br_\wi \prod_{\wi\in [m]\backslash (X_-\cup X_+\cup \{\wj\})} (1-\br_\wi),
\end{multline*}
where the sum runs over all pairs of subsets such that $|X_-|\ge |X_+|+2$.

Now, for any $s\ge 0$, $k\ge 2$ and any monomial 
$\prod_{\mu=1}^k \br_{j_\mu}\cdot\prod_{\nu=1}^s \br_{l_\nu}$, where 
$1\le j_1<\dots<j_k\le \wj-1$ and $\wj+1\le l_1<\dots < l_{s}\le m$, we get
the coefficient 
\[
\sum_{h=2}^k \sum_{u=0}^{\min\{h-2,s\}} (h-u-1) \binom{k}{h}\binom{s}{u}(-1)^{k-h+s-u}=:A(k,s)(-1)^{k+s}.
\]
Now, we only need to prove that $A(k,s)=\binom{s+k-2}{s}$. 

For notational convenience 
we will prove this statement for $k+1$. We start by splitting the first
binomial coefficient, and then  substituting $h\to h+1$ in the second part.
\begin{align*}
A(k+1,s)&=\sum_{h=2}^{k+1} (-1)^{h}\left( \binom{k}{h}+\binom{k}{h-1}\right) \sum_{u=0}^{\mins{h-2}} 
(-1)^{u} \binom{s}{u} (h-u-1) \\
&=A(k,s)+\sum_{h=1}^{k}(-1)^{h+1} \binom{k}{h} \sum_{u=0}^{\mins{h-1}} (-1)^{u} \binom{s}{u}(h-u)\\
&=A(k,s)+\sum_{h=1}^{k}(-1)^{h+1} \binom{k}{h} \sum_{u=0}^{\mins{h-2}}
(-1)^{u} \binom{s}{u}(h-u)\\
&\phantom{=A(k,s)\hskip3pt}
+\sum_{h=1}^k(-1)^{h+1}\binom{k}{h}(-1)^{h-1}\binom{s}{h-1}\\
&=\sum_{h=2}^{k}(-1)^{h+1} \binom{k}{h} \sum_{u=0}^{\mins{h-2}} (-1)^{u} \binom{s}{u}
+\sum_{h=1}^k\binom{k}{k-h}\binom{s}{h-1}.
\end{align*}
Using Lemma \ref{L:Identitet} below on the first sum and the Vandermonde
convolution 
$\sum_{i=0}^r\binom{x}{i}\binom{y}{r-i}=\binom{x+y}{r}$ on the last we get 
\[
A(k+1,s)=-\binom{s+k-1}{s+1} +\binom{s+k}{s+1}=\binom{s+k-1}{s}.
\]
The formula \eqref{tab} follows from the following lemma.

\begin{lemma} \label{L:Identitet} For any integers $k\ge 2$, $s\ge 0$ we have
\[\sum_{h=2}^{k}(-1)^{h} \binom{k}{h} \sum_{u=0}^{h-2} (-1)^{u} \binom{s}{u}=\binom{s+k-1}{s+1}
\]
\end{lemma}
\begin{proof} First note that by standard binomial identities we have 
\[ \sum_{u=0}^{h-2} (-1)^{u} \binom{s}{u}
= \sum_{u=0}^{h-2} \binom{-s+u-1}{u}
=\binom{-s+h-2}{h-2}
= (-1)^{h-2}\binom{s-1}{h-2}.
\]
We may now use the Vandermonde convolution to get, with $j=h-2$, 
\[
\sum_{h=2}^{k}(-1)^{h} \binom{k}{h} \sum_{u=0}^{h-2} (-1)^{u} \binom{s}{u}
=\sum_{j=0}^{k-2} \binom{k}{k-2-j} \binom{s-1}{j}=\binom{s+k-1}{k-2}.
\]
\end{proof}

This completes the proof of \refT{TA}.

\section{Proofs of corollaries}\label{SpfC}

\begin{proof}[Proof of \refC{C3}]
  When $m=3$, 
the double sum in \eqref{tab} is non-empty only if $i=3$; in this case there
is a single term with $s=0$ and $k=2$  and the result
follows immediately.

Alternatively, we can use \eqref{ta}:
$\si--\si+-1>0$ is possible only with $\si-=2$ and
  $\si+=0$; this requires that at least 2 states are larger than state $i$
  so $i=3$, and in this case the probability is 
  \begin{equation*}
\frac13\P\bigpar{S_3^--S_3^+=2}	
=\frac13\P\bigpar{I\qx3_1=I\qx3_2=1} 
=\frac13(p\xx 1-p\xx 3)(p\xx 2-p\xx 3).
\qedhere	
  \end{equation*}
\end{proof}

\begin{proof}[Proof of \refC{Cmin}]
The fact that $\ttq\xx m\ge \ttq\xx {m-1}\ge\dots$ follows from \eqref{ta}
and a simple 
coupling argument. Furthermore, if $i\le 2$, then $\si-\le1$ and
$q\xx i=0$.

For $i=m$, \eqref{tab} simplifies to
\begin{equation}\label{eri}
  \begin{split}
m\ttq\xx m&=\sum_{k=2}^{m-1}(-1)^k e_k((p\xx 1-p\xx m),\dots,(p\xx{m-1}-p\xx m))
%
\\&
=\prod_{j=1}^{m-1}\bigpar{1-(p\xx j-p\xx m)} + \sum_{j=1}^{m-1} (p\xx j-p\xx m)-1
\\&
=\prod_{j=1}^{m-1}\bigpar{1-(p\xx j-p\xx m)} -mp\xx m,
  \end{split}
\end{equation}
which is \eqref{cmin1}.

Furthermore, it follows from \eqref{cmin1} (or \eqref{ta})
that $\ttq\xx m$ will increase if
we 
decrease $p\xx m$ to 0 and simultaneously increase $p\xx 1$, say.
For $p\xx m=0$, the product in \eqref{eri} is largest when all $p_j$ for $j<m$
are equal, \ie{} $p\xx j=1/(m-1) $ for $j<m$; in this case 
\eqref{eri} yields $(1-1/(m-1))^{m-1}$.
Hence, for any $p\xx m>0$ and any $i$,
\begin{equation}
  q\xx i\le q\xx m < \frac1m\Bigpar{1-\frac{1}{m-1}}^{m-1}<\frac1m e^{-1}.
\end{equation}

Finally,
for any $i$,
\begin{equation}
  \E\bigpar{\si--\si+}
  =\sum_{p_j>p_i}(p_j-p_i)-\sum_{p_j<p_i}(p_i-p_j)
= \sum_{j}(p_j-p_i) = 1-mp_i.
\end{equation}
Consequently, using \eqref{ta},
\begin{equation}\label{magnus}
  \begin{split}
m\ttq_i=\E&\bigpar{\si--\si+-1}_+
\ge
\E\lrpar{\bigpar{\si--1}_+-\si+}
\\&
=\E\bigpar{S_i^--1+\ett{S_i^-=0}}-\E \si+
\\&
=\E \si--1+\P\bigpar{\si-=0}-\E \si+	
\\&
=\P\bigpar{\si-=0}-mp_i
=\prod_{p_j>p_i}\bigpar{1-(p_j-p_i)}-mp_i.
  \end{split}
\end{equation}
If $x\in\oi$, then $1-x\ge e^{-x}-\frac12x^2$, and it follows that
\begin{equation*}
  \prod_{p_j>p_i}\bigpar{1-(p_j-p_i)}
\ge \prod_{p_j>p_i}e^{-(p_j-p_i)}
-\frac12\sum_{p_j>p_i}(p_j-p_i)^2
\ge e^{-\sum_j p_j}
-\frac12\sum_{j}p_j^2,
\end{equation*}
which by \eqref{magnus} yields  the lower bound in \eqref{cmin2}.
\end{proof}

\begin{proof}[Proof of \refC{Ce}]
The expected number  is $\sum_j q_j$, and 
\eqref{cmin2} shows that $\sum_j q_j< \frac{m}{em}=e^{-1}$.

In the special case, 
\eqref{cmin2} shows that $q_i\sim e^{-1}/m$ for each small state,
and thus 
$\sum_j q_j\sim e^{-1}$.
\end{proof}

\begin{proof}[Proof of \refC{CP}]
  It is well-known and easy to see that if $I\sim\Be(r)$ and $Y\sim\Po(r)$,
  then $I$ and $Y$ may be coupled with $\E|I-Y|=2(e^{-r}-1+r)\le r^2$.
Couple in this way $I\qi_j$ with $Y_j\sim\Po(|p_i-p_j|)$, with the latter
independent, and define 
$\hS^+=\sum_{j:p_j<p_i} Y_j\sim\Po(\gl^+)$ and 
$\hS^-=\sum_{j:p_j>p_i} Y_j\sim\Po(\gl^-)$.
Then, by \eqref{ta}, \eqref{cp} and the triangle inequality,
\begin{equation*}
  \begin{split}
  m|\ttq_i-\hq_i|
&\le \E\bigabs{\si--\si+-(\hS^--\hS^+)}
\\&
\le \sum_{j:p_j>p_i}\E|I\qi_j-Y_j|
+\sum_{j:p_j<p_i} \E|I\qi_j-Y_j|
\\&
\le \sum_{j:p_j>p_i} (p_j-p_i)^2+\sum_{j:p_j<p_i} (p_i-p_j)^2.	
  \end{split}
\end{equation*}
This proves \eqref{cp}, which immediately yields \eqref{cp+}. To obtain
\eqref{cp-} we observe that the sum is
$\E\bigpar{\hS^++1-\hS^-}_+=-\E\bigpar{\hS^--\hS^+-1}_-$, 
and thus the difference between the sums
in \eqref{cp+} and \eqref{cp-} equals
\begin{equation*}
  \begin{split}
&\hskip-2em
\E\bigpar{\hS^--\hS^+-1}_++
\E\bigpar{\hS^--\hS^+-1}_-  
=
\E\bigpar{\hS^--\hS^+-1}
\\&=
\sum_{j:p_j>p_i} (p_j-p_i)-\sum_{j:p_j<p_i} (p_i-p_j)-1
=
\sum_{j=1}^m (p_j-p_i)-1
\\&
=1-mp_i-1
=-mp_i.
\qedhere
  \end{split}
\end{equation*}	
\end{proof}

\section{Proofs of results on average probabilities}\label{Spfrandom}

We let $\Da^m:=\{(x_1,\dots,x_m)\in \ooo^m:  \sum_ix_i=a\}$ and
$\D_{\le a}^m:=\{(x_1,\dots,x_m)\in \ooo^m: \sum_ix_i\le a\}$.
When integrating over $\Da^m$, we use the measure 
$\ddxq\=\dd x_1\dotsm \dd x_{m-1}$; this is thus the same as integrating
over $\D_{\le a}^{m-1}$ with Lebesgue measure, keeping $x_m=1-\sum_1^{m-1} x_i$.
Note that the volume of $\Da^m$ equals the volume of $\D_{\le a}^{m-1}$, i.e.\ 
$a^{m-1}/(m-1)!$. Hence, the uniform probability measure on $\D^m=\D_1^m$ is
$(m-1)!\,\ddxq$.

More generally, we have the well-known Dirichlet integral
\begin{equation}
  \label{dir}
\int_{\Da^m}x_1^{\ga_1-1}\dotsm x_m^{\ga_m-1}\ddxq
=a^{\ga_1+\dots+\ga_m-1}\frac{\prod_{i=1}^m\gG(\ga_i)}{\gG\bigpar{\sum_{i=1}^m\ga_i}}
\end{equation}
for any $\ga_1,\dots,\ga_m>0$. (For $m=2$ this is the standard Beta integral
and the general case follows easily by induction.
An alternative, probabilistic, standard proof is to 
let $T_1,\dots,T_m$ be independent $\Exp(1)$ variables and 
evaluate 
$\E(T_1^{\ga_1-1}\dotsm T_m^{\ga_m-1})$ by conditioning on $T_1+\dots+T_m$.)

\begin{proof}[Proof of \refT{T:expected}]
Recall from Corollary \ref{Cmin} that if we assume $p_m\le p_i$ for all
$i$ and let $r_i:=p_i-p_m$, 
then 
$q_m=\frac{1}{m}\prod_{i=1}^{m-1} (1-r_i) -p_m$. Choosing a vector 
$(p_1,\dots,p_m)$ uniformly from
$\D_1^m$, there are $m$ possibilities for the position of the minimum
coordinate; by symmetry it suffices to consider the case when $p_m$ is 
the minimum (multiplying below by a factor $m$). 
Then the vector $(r_1,\dots,r_{m-1})$ is uniformly distributed in
$\D_{1-mp_m}^{m-1}$ and $\ddxq=\dd p_m\dd\bfr$.
We thus get
\begin{equation}\label{expected}
  \begin{split}
\E q\xx m=&m(m-1)!\int_0^{1/m} \int_{{\bfr}\in\D_{1-mp_m}^{m-1}} 
\left(\frac1m\prod_{i=1}^{m-1}(1-r_i) - p_m  \right)\dd\bfr\, \dd p_m
.	  \end{split}
\end{equation}
We treat the two terms in the bracket separately.
First, using symmetry and \eqref{dir},
\begin{equation}\label{texp1}
  \begin{split}
(m-1)!&\int_0^{1/m} \int_{{\bfr}\in\D_{1-mp_m}^{m-1}} \prod_{i=1}^{m-1}(1-r_i)
\dd\bfr\, \dd p_m\\ 
&=(m-1)!
\sum_{k=0}^{m-1}\binom{m-1}k (-1)^k
\int_0^{1/m} \int_{{\bfr}\in\D_{1-mp_m}^{m-1}} \prod_{i=1}^{k}r_{i}	
\dd\bfr\, \dd p_m\\ 
&=
\sum_{k=0}^{m-1}\binom{m-1}k (-1)^k
\int_0^{1/m} (1-mp_m)^{k+m-2}\frac{(m-1)!}{\gG(k+m-1)} \dd p_m\\ 
&=
\sum_{k=0}^{m-1}\binom{m-1}k (-1)^k
\frac{(m-1)!}{m\gG(k+m)} \\ 
&=
\frac1m
\sum_{k=0}^{m-1}(-1)^k\binom{m-1}k \frac{1}{ m\rise k} .
  \end{split}
\raisetag{1.5\baselineskip}
\end{equation}
Similarly, the second term becomes
\begin{equation}\label{texp2}
  \begin{split}
m(m-1)!&\int_0^{1/m} \int_{{\bfr}\in\D_{1-mp_m}^{m-1}} p_m
\dd\bfr\, \dd p_m
=m(m-1)
\int_0^{1/m} (1-mp_m)^{m-2} p_m \dd p_m\\ 
&=\frac{m-1}{m}
\int_0^{1} (1-x)^{m-2} x \dd x
=\frac1{m^2}.  
\end{split}
\end{equation}
The formula \eqref{texp} follows from \eqref{texp1} and \eqref{texp2},
noting that the first two terms in the first sum equal
$1-\frac{m-1}m=\frac1m$, which cancels the term $-1/m^2$.

The asymptotic expansion \eqref{texpoo} is easy to deduce from
\eqref{texp}. 
(One can also easily obtain further terms.)

Since $q\xx m\le e\qw/m$ a.s., 
we have $\E|e\qw-mq\xx m|=\E(e\qw-mq\xx m)=e\qw-m\E
q\xx m\to0$, and thus $mq\xx m\pto e\qw$.
\end{proof}

\begin{proof}[Proof of \refT{TE}]
Take $i=m$.
We take expectations in \eqref{tabb} and obtain by symmetry
\begin{multline}\label{tea}
\E q_m=
\frac{1}{m}\sum_{s=0}^{m-3}\sum_{k=2}^{m-s-1}(-1)^{s+k}\binom{s+k-2}{s} 
\binom{m-1}{s,k,m-1-s-k} 
\\\cdot
\E\lrpar{ \prod_{j=1}^s(p_m-p_{j})_+ \prod_{l=1}^k(p_{s+l}-p_{m})_+}  .
\end{multline}

We use the standard method of generating uniform $\ppm$ on $\gD^m$ by taking
\iid{} exponential random variables $T_1,\dots,T_m\sim\Exp(1)$ and letting
$p_i\=T_i/S_m$ with $S_m\=\sumim T_i$.
Recall that then $(\ppm)$ and $S_m$ are independent, 
and that $S_m$ has a Gamma$(m,1)$ 
distribution with 
$\E S_m^\ga=\gG(m+\ga)/\gG(m)$ for $\ga\in\bbN$.
Hence, by conditioning on $T=T_m$ and using independence, 
\begin{equation}\label{teb}
  \begin{split}
\E\biggpar{ &\prod_{j=1}^s(p_m-p_{j})_+ \prod_{l=1}^k(p_{s+l}-p_{m})_+}  
\\&
=
\frac{
\E\lrpar{ \prod_{j=1}^s(T_m-T_{j})_+ \prod_{l=1}^k(T_{s+l}-T_{m})_+} } 
{\E S_m^{s+k}}
\\
&=\frac{\gG(m)}{\gG(m+k+s)}
\E\lrpar{ \bigpar{\E((T-T_{1})_+\mid T)}^s \bigpar{\E((T_{2}-T)_+\mid T)}^k } 
.	
  \end{split}
\end{equation}
For any $t>0$ and $j\ge1$ we have
\begin{align}\label{et+}
\E(t-T_j)_+	&=\int_0^t (t-x)e^{-x}\dd x = e^{-t}-1+t  ,
\\
\E(T_j-t)_+	&=\int_t^\infty (x-t)e^{-x}\dd x= e^{-t}. \label{et-}
\end{align}
Hence, the final expectation in \eqref{teb} equals
\begin{equation}\label{tec}
  \begin{split}
\E\Bigpar{ &\bigpar{e^{-T}-1+T}^s e^{-kT} }
=
\intoo \bigpar{e^{-t}-1+t}^s e^{-(k+1)t} \dd t
\\&
= \intoo 
\sum_{i+j\le s}
\binom{s}{i,j,s-i-j} e^{-it}t^j(-1)^{s-i-j} e^{-(k+1)t} \dd t
\\&
=\sum_{i+j\le s}  (-1)^{s-i-j}
\binom{s}{i,j,s-i-j} \intoo t^j e^{-(i+k+1)t} \dd t
\\&
=\sum_{i+j\le s}  (-1)^{s-i-j}
\binom{s}{i,j,s-i-j} \frac{j!}{(i+k+1)^{j+1}}.
  \end{split}
\end{equation}
The result now follows from \eqref{tea}--\eqref{tec}.
\end{proof}

\begin{proof}[Proof of \refT{Too}]
Note first that the expectation in \eqref{xcp} can be evaluated as in
\eqref{xcp+}--\eqref{xcp-}. For \eqref{xcp+} this is immediate; for
\eqref{xcp-} it follows from 
$$
\E\bigpar{\hS^--\hS^+-1}_+-\E\bigpar{\hS^--\hS^+-1}_-
=
\E\bigpar{\hS^--\hS^+-1}=\gl^--\gl^+-1.
$$

We let, as in the proof of \refT{TE}, 
$p_i\=T_i/S_m$ with $T_i\sim\Exp(1)$ \iid{} 
and $S_m\=\sumim T_i$.
Consider state 1 and condition on $T_1$, leaving $T_2,T_3,\dots$ \iid{}
$\Exp(1)$. We define $\gl^\pm\=\sum_j(p_1-p_j)_\pm$ as in \refC{CP} (with
$i=1$). 
As \mtoo, the law of large numbers shows that a.s., 
using \eqref{et+}--\eqref{et-},
\begin{align}\label{sjw}
  \frac{S_m}m&
=\frac{T_1}m+\sum_{j=2}^m\frac{T_j}{m}
\to0+\E T_2=1,
\\ \notag
\gl^+&=\sum_{j=2}^m\frac{(T_1-T_j)_+}{S_m}
=\frac{m}{S_m}\sum_{j=2}^m\frac{(T_1-T_j)_+}{m}
\\& \label{sjw+}
\to \E\bigpar{(T_1-T_2)_+\mid T_1} = e^{-T_1}-1+T_1,
\\ \notag
\gl^-&=\sum_{j=2}^m\frac{(T_1-T_j)_-}{S_m}
=\frac{m}{S_m}\sum_{j=2}^m\frac{(T_1-T_j)_-}{m}
\\& \label{sjw-}
\to \E\bigpar{(T_1-T_2)_-\mid T_1} = e^{-T_1},
\\ \label{sjw2}
\sum_{j=1}^m{(p_j-p_1)^2}
&=\frac{m}{S_m^2}\sum_{j=2}^m\frac{(T_j-T_1)^2}{m}
=O\Bigparfrac1{m}.
\end{align}

First, by \eqref{sjw}, 
\begin{equation}\label{maap}
  mp_1=\frac{m}{S_m}T_1\to T_1.
\end{equation}

Next, we apply \refC{CP}, noting that $\hq_1=\Phi(\gl^-,\gl^+)/m$.
Hence, by \eqref{cp2} and \eqref{sjw2}, 
$m q_1-\Phi(\gl^-,\gl^+)\to0$ a.s.
For any $\gl_1,\gl_2>0$, we can couple $\hS_1\sim\Po(\gl_1)$ and
$\hS_2\sim\Po(\gl_2)$ such that $\E|\hS_1-\hS_2|\le|\gl_1-\gl_2|$, and it follows
from \eqref{xcp} and the triangle inequality that
$|\Phi(\gl_1^-,\gl_1^+)-\Phi(\gl_2^-,\gl_2^+)|\le|\gl_1^--\gl_2^-|+|\gl_1^+-\gl_2^+|$
for any $\gl_1^-,\gl_1^+,\gl_2^-,\gl_2^+$.
Hence \eqref{sjw+}--\eqref{sjw-} imply that $\Phi(\gl^-,\gl^+)\to \Psi(T_1)$
a.s. 
Consequently, a.s.
\begin{equation}\label{maaq}
  mq_1\to \Psi(T_1).
\end{equation}
The limit \eqref{too} follows from \eqref{maap}--\eqref{maaq}.
Since $mq_1\le e\qw$ a.s. by \refC{Cmin}, \eqref{tooe} follows by dominated
convergence toghether with \eqref{xcp+}.
\end{proof}

\section{Further comments}

\subsection{Several states at once}\label{SSseveral}

Several states may suffer from the Alabama paradox at the same time. A simple
example is given in \refF{Falabama2}.

\begin{figure}
\begin{tabular}{|c|r|r|r|}
\hline
state & pop. & $\mu_i$ & seats \\
\hline
A & 28 & 1.40 &  1 \\
B & 27 & 1.35 &  1 \\
C & 27 & 1.35 &  1 \\
D & 9 & 0.45 & \bf1 \\
E & 9 & 0.45 & \bf1 \\
\hline
sum & 100 & 5.00 & 5\\
\hline
\multicolumn{4}{l}{5 seats}
\end{tabular}
\qquad
\begin{tabular}{|c|r|r|r|}
\hline
state & pop. & $\mu_i$ & seats \\
\hline
A & 28 & 1.68 &  \bf2 \\
B & 27 & 1.62 & \bf2  \\
C & 27 & 1.62 & \bf2 \\
D & 9 & 0.54 & 0 \\
E & 9 & 0.54 & 0 \\
\hline
sum & 100 & 6.00 & 6\\
\hline
\multicolumn{4}{l}{6 seats}
\end{tabular}
\caption{A double Alabama paradox for states D and E. 
Numbers in boldface are rounded up.}
\label{Falabama2}
\end{figure}
This can be analysed in the same way using the methods in \refS{Spf}. 
The typical case is when two states $i$ and $j$ are in positions 0 and 1,
and both are overtaken by three other runners. However, other similar
configurations are possible, and we leave the details to the reader except
for two simple examples.

\begin{example}
  Suppose that there are 5 states, with $p_1\ge p_2\ge p_3\ge p_4\ge p_5$
and $(p_i)$ \linQ.
If a double Alabama paradox occurs, it has to be for states 4 and 5, 
with one of them in position 0, the other in position 1, and the three
others overtaking both of them.
Letting $x$ be the distance between the runners 4 and 5, we find the
probability, using uniform distribution as before,
\begin{equation*}
  \begin{split}
	&\frac15\int_0^{p_3-p_4}(p_3-p_4-x)(p_2-p_4-x)(p_1-p_4-x)\dd x
\\&+
	\frac15\int_0^{p_4-p_5}(p_3-p_4)(p_2-p_4)(p_1-p_4)\dd x
\\&+
	\frac15\int_{p_4-p_5}^{p_3-p_5}(p_3-p_5-x)(p_2-p_5-x)(p_1-p_5-x)\dd x.
  \end{split}
\end{equation*}
The integrals are easily evaluated (preferably by computer), but the
resulting polynomial in $p_1,\dots,p_5$ does not look particularly nice or
illuminating so we omit it.
In the extreme case $p_1\approx p_2\approx p_3 \approx 1/3$,
$p_4\approx p_5\approx 0$, the probability becomes $1/810$.
\end{example}

\begin{example}
  \label{Efluid}
Consider the extreme case in \refC{Ce}, with $m-o(m)$ small states and
$o(m)$ medium-size states with the bulk of the population.
In this case, the flag is most likely carried by a small state, say $i$.
The probability that it is overtaken by a large state $j$ is $p_j-p_i\approx
p_j$, and we can approximate the distributions of the number $M$ of states
overtaking it by a Poisson distribution with mean
$\sum_j(p_j-p_i)_+\approx\sum_j p_j=1$. 
Thus $i$ loses a seat if $M>1$.
Furthermore, since there are very many
small states, their runners are very narrowly spaced, and most likely the
runners in positions $1,\dots,M-2$ are also small states, and are also
overtaken by the same $M$ states as $i$; in this case they all lose a seat.

Consequently, if the random variable 
$X_m$ is the number of states suffering from the Alabama paradox
when $n$ is increased (from a random value), then as $m\to\infty$ (in this
case), $X_m\dto X\=(Y-1)_+$ with
$Y\sim\Po(1)$. Note that 
$\P(X>0)=\P(Y>1)=1-2e^{-1}$ and 
$\E X=\E( Y-1+\ett{Y=0})\allowbreak=\P(Y=0)=e^{-1}$.
\end{example}

\subsection{The linearly dependent case}\label{SSdep}
The proof above uses Weyl's theorem and thus requires that $\ppm$ are
\linQ; indeed, as said in \refS{S:intro}, simple examples shows that
\refT{TA} does not hold for arbitrary $p_i$.
We note also that some of the corollaries  might be far from true 
in the case of rational relative sizes, see Proposition \ref{P:highE} below.
Nevertheless, it is possible to use much of the argument above also for 
the linearly dependent case. We sketch this below, leaving many details to
the reader.

Note first that in general it may happen that the remainders
$\mu_i-\floor{\mu_i}$ happen to be equal for two or more states, and 
it may be necessary to round up one or several of these and round down the
others; in this case, the choice is determined by lot.
A simple example is given in \refF{Fties}; note that state C  may suffer the 
Alabama paradox either when increasing from 4 to 5 seats or from 5 to 6,
but not at both times; the probability is $1/4$ each time. 
For the asymptotic analysis this complication is no real problem, however, since
it suffices to consider the expectation $\E\nu_i(N)$, \ie, the 
sum over $n\le N$ of the probability of an Alabama paradox at time
$n$. Indeed, although the example in \refF{Fties} shows that there may be a
dependency between the occurrence of the Alabama paradox for some number $n$
of seats and the next number $n+1$, more distant occurrences are independent;
\ie, the random sequence of occurrences of the paradox is 1-dependent.  
Hence the variance of the total number is $O(N)$, and by considering 
odd and even $n$ separately, which yields two subsequences of independent
random indicators, it is easily seen that $(\nu_i(N)-\E\nu_i(N))/N\to0$
almost surely.  

\begin{figure}
\begin{tabular}{|c|r|r|r|r|r|r|r|r|}
\hline
state & pop. & 1 & 2&3&4&5&6&7\\
\hline
A & 6 & \bf0.6 & 1.2 & \bf1.8 & \it2.4 & 3.0 & \it3.6 & 4.2 \\
B & 3 & 0.3 & \bf0.6 & \bf0.9 & 1.2 & \it1.5 & \bf1.8 & 2.1 \\
C & 1 & 0.1 & 0.2 & 0.3 & \it0.4 & \it0.5 & \it0.6 & \bf0.7 \\
\hline
\end{tabular}
\caption{Ties. Numbers in boldface are rounded up; numbers in italics may be
  rounded up or down, as determined by lot.}
\label{Fties}
\end{figure}

\begin{remark} 
  If two or several states have exactly the same population, it may be
  necessary to draw lots between them for both $n$ seats and $n+1$. If we do
  this independently, then some state might lose a seat by being unlucky the
  second time. This obvious consequence of drawing lots is not an example of
  the Alabama paradox and should be disregarded. (For example, we may list
the states with the same population in some random order, once and for all,
and use this as a priority list for each $n$.)
\end{remark}

Let $\pp\=(p_1,\dots,p_m)$.
 We now use $q_i$ for the true probability and $\tq_i$ for the value in
\eqref{ta} for any \ppm.
Consider again, for notational convenience, state 1. 
The argument in \refS{Spf} shows that, with $L_n$ given by \eqref{l}
(noting that the actual position may be different when there are ties), 
for some functions $f$, $g$ and $h$ and with $z_j=(p_j-p_1)/m$,
for simplicity assuming that $N$ is a multiple of $m$, 
\begin{align}
 \E\nu_1(N)&=\sum_{n=0}^{N-1}
 f\bigpar{\fract{n(p_2-p_1)},\dots,\fract{n(p_m-p_1)},L_n}
\notag
\\
&=\sum_{n=0}^{N-1}
g\bigpar{\fract{nz_2},\dots,\fract{nz_m},\Modm(n)}.
\notag
\\
&=\sum_{k=1}^{N/m}
 mh\bigpar{\fract{k(p_2-p_1)},\dots,\fract{k(p_m-p_1)}}.
\label{deph}
\end{align} 
(To see the last equality, write $n=(k-1)m+l$ and 
define 
$h(x_2,\dots x_m)\=
\sum_{l=0}^{m-1}g(\fract{x_2+(l-m)z_2},\dots,\fract{x_m+(l-m)z_m},l)/m$.)
\refL{LW} does not apply when
$p_1,\dots,p_m$ are linearly dependent over $\bbQ$, but the proof of it sketched
above shows that the sequence 
$\bigpar{\fract{n(p_2-p_1)},\dots,\fract{n(p_m-p_1)}}$ is \ud{} on a subgroup
of $[0,1)^{m-1}$; more precisely, the empirical distributions converge to 
the uniform probability measure $\mu_\pp$ on this
subgroup, which has Fourier coefficients given by
\begin{equation}\label{mupp}
\widehat\mu_\pp(a_2,\dots,a_m)=
\begin{cases}
  1 & \text{ if } \sum_{j=2}^ma_j(p_j-p_1)\in\bbZ,\\
0 & \text{ otherwise}.
\end{cases}
\end{equation}
The functions $f$, $g$ and $h$ are linear combinations of products of
indicators and are a.e.\ continuous; moreover, $h$ is $\mu_\pp$-a.e.\
continuous. Hence,   \eqref{deph} shows that 
$\E\nu_1(N)/N\to \int h\dd\mu_\pp$, 
showing the existence of the limit $q_1=\int h\dd\mu_\pp$ in general.
Note that in this notation, the value $\tq_1$ in \eqref{ta} is 
$\tq_1=\int h\dd\mu$, 
where $\mu$ is the uniform distribution on $[0,1)^{m-1}$. 

The functions $f$, $g$ and $h$ depend on $\pp$, \cf{} \eqref{i+}--\eqref{i-},
but if we write $h_\pp$, then $h_{\pp_k}\to h_\pp$ \aex{} for any
sequence of population distributions (for a fixed number of states)
$\pp_1,\pp_2,\dots,$ with $\pp_k\to\pp$. 
Let $\pp_k=(p_{1k},\dots,p_{mk})$.
If further,
for every integer vector $(a_1,\dots,a_m)\neq0$, we have 
$\sum_{i=1}^m a_ip_{ik}\neq0$ for all large $k$,
then it follows from \eqref{mupp} that 
$\widehat\mu_{\pp_k}(a_2,\dots,a_{m})\to0$ for all $(a_2,\dots,a_m)\neq0$,
and thus $\mu_{\pp_k}\to\mu=\mu_\pp$. It follows, \eg{} by 
\cite[Theorem 5.5]{Billingsley} that 
$q_{1,k}=\int h_{\pp_k}\dd\mu_{\pp_k}\to\int h_\pp\dd\mu_\pp$
and also
$\int h_{\pp_k}\dd\mu\to\int h_\pp\dd\mu_\pp$ so
$q_{1,k}-\int h_{\pp_k}\dd\mu\to0$.
The claim in \refR{Rind} now follows, since otherwise one could, for some
$\eps>0$, find such a sequence $\pp_k$ of distributions with  
$|q_{1,k}-\int h_{\pp_k}\dd\mu|\ge\eps$ and (by taking a subsequence)
$\pp_k\to\pp$ for some $\pp$; a contradiction.

\smallskip
We end with a couple of counterexamples in the rational case. 
The upper bound  $q_i<\frac{1}{m}e^{-1}$ in \refC{Cmin} is not true in
general, not even for $m=3$. 
As an easy example one 
may study three states with $p_1=p_2=3/7$ and $p_3=1/7$, where
$q_3=1/7>\frac{1}{3e}$, since the smallest state suffers when $n$ increases
from 3 to $4\pmod 7$.

Moreover,  the upper bound $1/e$ on the expected number of states
suffering from the paradox in \refC{Ce} 
  is not true in general. Consider the case when $p_1=\dots=p_6=1/7$ and $p_7=\dots=p_{106}=1/700$.
When \eg{} the number of seats changes from 42 to 43 all the 6 large states
will change from 6 to 7 seats. 
Thus 5 of the small states will suffer from the Alabama paradox. The paradox
will happen 98 times during the period of length 700; 
90 of these 5 small states will suffer, but the number is smaller close to
the beginning and end of the period
($n=7$, 14, 21, 28 or 671, 678, 685, 692 $\pmod {700}$) and 
the expected number of states to suffer from the Alabama paradox for a
random number of seats is  
\begin{equation*}
\frac{90\cdot 5 + 2\cdot(1+2+3+4)}{700}=\frac{47}{70},
\end{equation*}
which is much larger than $1/e$. This can be generalized in the following way. 

\begin{proposition}\label{P:highE}
Let $x\ge2$ and $y$ be positive integers, with $x$ relatively prime to $y-1$
and $x^2-3x<y$.
Assume that the number of states is $m=y+x-1$ with 
relative sizes $p_1=\dots=p_{x-1}=\frac{1}{x}$ and
$p_x=\dots=p_m=\frac{1}{x\cdot y}$.  
Then the expected number of states suffering from the Alabama paradox equals
\begin{equation}
    \label{big}
\frac{(x-2)(y-x+1)}{x y},
\end{equation}
which in particular can be made arbitrarily close to $1$.
\end{proposition}
\begin{proof} The values of $n\pmod {xy}$
where the Alabama paradox might occur are, \eg{} by considering runners as
in \refS{Spf}, 
$a x+b$, for integers $1\le a\le y-2$ and $0\le b\le x-1$ such that 
$ax+b>by$ and $ax+b+1<(b+1)y$, \ie{}
$b=\floor{ax/(y-1)}$. There are $y-2$ such values $ax+b$, one for each
$a\in\set{1,\dots,y-2}$.  

If $x-2\le a\le y-x+1$, then for $n=ax+b$, with $b$ as above, the $x-1$
large states will get $a$ seats each and the remaining $a+b$ seats will go
to $a+b$ of the small states. For $n=ax+b+1$, the larger states will get
$a+1$ seats each and only $a+b-x+2$ seats are left to the small states; thus
$x-2$ small states will suffer from the Alabama paradox.
If $1\le a\le x-3$, then $b=0$, and the number of small states receiving a
seat will drop from $a+b=a$ to 0 as $n$ increases from $ax+b$ to $ax+b+1$.
Finally, if $y-x+2\le a\le y-2$, then $b=x-1$ and the number will drop from
$y$ to $a+b-x+2=a+1$, so $y-1-a$ states will suffer.
Summing these numbers give a total of 
$(y-2x+4)(x-2)+2\sum_{i=1}^{x-3}i=(y-x+1)(x-2)$ states suffering
in the period $xy$.

Taking \eg{} $y=x^2$ and letting $x\to\infty$, the expected number
\eqref{big} tends to 1.
\end{proof}
We thank Warren D. Smith for asking a question that made us produce Proposition \ref{P:highE}.
It would be interesting to see a proof of 1 being a general upper bound or
an example to the contrary.

\newcommand\AMS{Amer. Math. Soc.}
\newcommand\Springer{Springer-Verlag}
\newcommand\Wiley{Wiley}

\newcommand\vol{\textbf}
\newcommand\jour{\emph}
\newcommand\book{\emph}
\newcommand\inbook{\emph}
\def\no#1#2,{\unskip#2, no. #1,} 
\newcommand\toappear{\unskip, to appear}

\newcommand\urlsvante{\url{http://www.math.uu.se/~svante/papers/}}
\newcommand\arxiv[1]{\url{arXiv:#1.}}
\newcommand\arXiv{\arxiv}

\def\nobibitem#1\par{}

\end{document}

\end{document}